\documentclass{article} 

\usepackage{amsmath,amsthm,amssymb,amsfonts}
\usepackage{graphicx}
\usepackage{geometry} % to change the page dimensions
\newtheorem{Theorem}{Theorem}[section]
\theoremstyle{plain}

\newtheorem{Proposition}[Theorem]{Proposition}

\numberwithin{equation}{section}

\title{Sample genealogies within a Brownian excursion}
\author{Sergey Bocharov and Simon C. Harris}
\begin{document}
\maketitle
\begin{abstract}
In this article, we apply It\^o's excursion theory to find the joint law of minima of a Brownian excursion conditioned to go above level $1$ over the intervals defined by $k$ independent identically distributed (i.i.d.) points of intersection of the excursion with level $a \in (0,1]$.

%This approach offers a quick way to find the limiting joint law of split times of an i.i.d. sample of particles alive at time $at$ in a critical Galton-Watson tree conditioned to survive to time $t$.

This approach offers a intuitive way to find the joint law of coalescent times of an i.i.d. sample of k particles alive at time $a$ in Aldous' continuum random tree of height $1$. This can be thought of as \emph{sampling of the limit} of critical Galton-Watson trees conditioned to survive a large time, agreeing with some special cases obtained in \cite{HJR20} and \cite{HHKP24} which instead consider \emph{the limit of sampling} from critical Galton-Watson trees.
\end{abstract}

%====================================================================================================
\section{Introduction and Main Results}

Our object of study will be the Brownian excursion from level $0$ conditioned to go above level $1$. One way to construct it (see, for example, \cite{NP89}) is to take a standard Brownian motion $B$, let $\sigma$ be the time of the last visit to $0$ before $B$ hits level $1$ for the first time and let $\tau$ be the time of the first visit to $0$ after $B$ hits level $1$ for the first time.
Then the process
\begin{equation}
\label{x}
X_t := B_t \mathbf{1}_{[\sigma, \tau]}(t) 
\end{equation}
is a Brownian excursion from level $0$ conditioned to go above level $1$. 
\begin{figure}[!htbp]
\begin{center}
\includegraphics[scale=0.75]{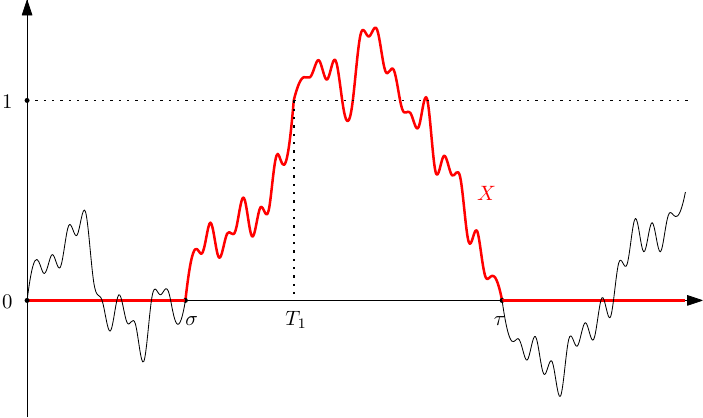}
\caption{A Brownian excursion conditioned to go above level $1$.}
\end{center}
\end{figure}

The theorem on the next page is our main result.
\begin{Theorem}
\label{main}
Let $X_t$, $t \geq 0$ be a Brownian excursion from level $0$ conditioned to go above level $1$ 
and, for $a \in (0,1]$, let $L^a_t$, $t \geq 0$ be the local time at level $a$ of the process $X_t$, $t \geq 0$.
Suppose that $U_1$, ..., $U_k$ are random variables which, conditional on $(L^a_t)_{t \geq 0}$, are independent with distribution
\[
\mathbb{P} \Big( U_1 \leq t \Big\vert \big( L^a_t \big)_{t \geq 0} \Big) = \frac{L^a_t}{L^a_\infty},
\]
where $L^a_\infty$ is the total amount of local time at level $a$ accumulated by $X$. 
Then $U_1$, ..., $U_k$ represent the times of a uniform sample of $k$ intersections of level $a$ by $X$.  
% with replacement  
Further, define 
\[
S_i := \min_{t \in [U_{(i)}, U_{(i+1)}]} X_t
\]
where $U_{(1)}$, ..., $U_{(k)}$ are the order statistics of $U_1$, ..., $U_k$.

Then the joint distribution of $S_1$, ..., $S_k$ can be expressed in the form 
\begin{align}
\label{tail_a}
\mathbb{P} \big(S_1 \!>\! x_1, \cdots, \ S_{k-1} \!>\! x_{k-1}\big) = 
&k \int_0^\infty \frac{1}{(1+ as)^2} \Big[ \prod_{i=1}^{k-1} \frac{(a-x_i)s}{1+(a-x_i) s} \Big] \mathrm{d}s\nonumber\\
&\qquad -k \int_0^\infty \frac{1-a}{(1+ as)^2} \Big[ \prod_{i=1}^{k-1} \frac{(a-x_i)s}{1-x_i +(a-x_i) s} \Big] \mathrm{d}s,
\end{align}
where $x_1$, ..., $x_k \in [0,a]$. In particular, the joint probability density of $S_1$, ..., $S_k$ is 
given by
\begin{align}
\label{density_a}
f_{S_1, ..., S_{k-1}}(x_1, ..., x_{k-1}) 
= &k \int_0^\infty  \frac{1}{(1+as)^2} \bigg[ \prod_{i=1}^{k-1} \frac{s}{(1+(a-x_i) s)^2} \bigg]\mathrm{d}s \nonumber\\
&\qquad -k \int_0^\infty  \frac{1-a}{(1+as)^2} \bigg[ \prod_{i=1}^{k-1} \frac{(1-a)s}{(1-x_i+(a-x_i) s)^2} \bigg]\mathrm{d}s,
\end{align}
where $x_1$, ..., $x_k \in [0,a]$.
\begin{figure}[!htbp]
\begin{center}
\includegraphics[scale=0.95]{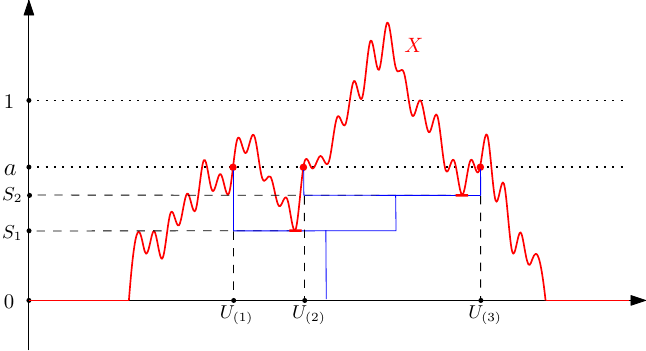}
\caption{Minima of a Brownian excursion between three sampled points (and an embedded genealogical  tree in blue).}
\end{center}
\end{figure}
\end{Theorem}

We will prove this result in Section 3 making use of It\^o's excursion theory (a good account of which is given by Rogers in \cite{R89}). The proof will require some identities about a sample of i.i.d. random variables which are uniform on an interval of an exponentially-distributed length, which we state in Section 3 and prove separately in Section 4.

Our work is motivated by the correspondence between Brownian excursions and critical branching processes, which offers an efficient way of studying genealogies of such branching processes (see,  for example, \cite{L10}, \cite{NP89} and \cite{P04}), and also some recent results on genealogies in critical branching processes from \cite{HJR20} and \cite{HHKP24}. This will be discussed in Section 2 below.

%====================================================================================================
\section{Motivation}

A continuous-time Galton-Watson process is a stochastic process that starts with a single particle at time $t=0$. This particle lives for a random time with the $Exp(\beta)$ distribution for some $\beta >0$. When it dies it produces a random number $\xi$ of new particles according to the offspring distribution 
\[
\mathbb{P} (\xi = k) = p_k \ , \quad k = 0, \ 1, \ 2, \ ...
\]
From this point onward all newly-born particles (if there are any) replicate the behaviour of the initial particle independently of each other. We let 
\[
m := \mathbb{E} \xi = \sum_{k \geq 0} k p_k
\]
denote the mean of the offspring distribution. Assuming $m < \infty$, the process is said to be supercritical, critical or subcritical whenever $m>1$, $m=1$ or $m<1$ respectively. 

We are only interested in the critical case. We are also going to assume throughout this article that the variance of the offspring distribution, which we denote by $v^2$ is finite:
\[
v^2 := \mathbb{E} (\xi^2) - \big( \mathbb{E} \xi \big)^2= \sum_{k \geq 0} k^2 p_k - 1 < \infty.
\] 
It has been known for a long time from early works of Kolmogorov and Yaglom that if $N_t $ is the size of the population at time $t$ then $\mathbb{P} (N_t > 0) \sim \frac{2}{v^2 t}$ and 
\begin{equation}
\label{exp_limit}
\frac{2}{v^2} \cdot \frac{N_t}{t} \ \Big\vert \{ N_t>0 \} \Rightarrow Exp (1)
\end{equation}
as $t \to \infty$.

The discrete-time Galton-Watson process is defined in a similar way except that particles live for a deterministic one unit of time, but this  doesn't really change the asymptotic behaviour of the system. 

There has been a long interest in the study of  limiting distributions (as $t \to \infty$) of split times (or times of the last common ancestor) of particles sampled from the critical Galton-Watson tree conditioned to survive to time $t$. This goes back to an old work of Zubkov \cite{Z75} where he established the time of the last common ancestor of all particles alive at time $t$ in such a tree. For a recent overview of this topic we refer to \cite{HJR20}. Let us mention two recent results that motivated our work. The first result can be found in sections 2.3 and 6.6 of \cite{HJR20}.
 
 \begin{Theorem}[S. Harris, S. Johnston, M. Roberts (2020)]
 \label{HJR}
Consider a continuous-time critical Galton-Watson process conditioned on $\{ N_t \geq k\}$. If $k$ distinct particles are sampled uniformly from those alive at time $t$ and $S_1(t)$, ..., $S_{k-1}(t)$ are the split times of these particles then the random vector 
\[
\Big( \frac{S_1(t)}{t}, \cdots, \frac{S_{k-1}(t)}{t} \Big)
\]
converges in distribution as $t \to \infty$ to a random vector $(S_1, \cdots, S_{k-1})$ with probability density 
\begin{equation}
\label{thm1}
f_{S_1, ..., S_{k-1}} (x_1, \cdots x_{k-1}) = k \int_0^\infty  \frac{1}{(1+s)^2} \bigg[ \prod_{j=1}^{k-1} \frac{s}{(1+(1-x_j) s)^2} \bigg]\mathrm{d}s.
\end{equation}
\end{Theorem}
Note that \eqref{thm1} agrees with our general formula \eqref{density_a} if we let $a=1$.
\begin{figure}[htbp]
\begin{center}
\includegraphics[scale=0.9]{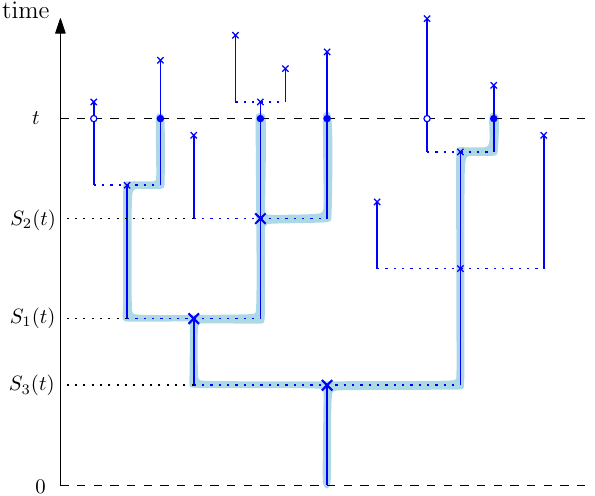}
\caption{Split times of four particles sampled at time $t$.}
\end{center}
\end{figure}

The second result can be found in Proposition 4.3 of \cite{HHKP24} (see also Step 7 of the proof of Proposition 4.3).
\begin{Theorem}[S. Harris, E. Horton, A. Kyprianou, E. Powell (2024)]
\label{HHKP}
Consider a continuous-time critical Galton-Watson process conditioned on $\{ N_t > 0\}$. Let $S(t)$ be the split time of two particles: one chosen uniformly at random from those alive at time $t$ and another one chosen independently uniformly at random from those alive at time $at$, where $a \in (0,1)$. Then 
\[
\frac{S(t)}{t} \Rightarrow S
\]
as $t \to \infty$ where the random variable $S$ has probability density of the form
\begin{equation}
\label{thm2a}
f_{S}(x) = \frac{2}{a^2} \int_0^\infty \frac{s}{(1+s)^2 \big(1 + (1 - \frac{x}{a})s\big)^2} - \frac{s}{\big(1+s+\frac{a}{1-a}\big)^2 \big(1+(1-\frac{x}{a})(s+\frac{a}{1-a})\big)^2} \mathrm{d}s
\end{equation}
which integrates to 
\begin{equation}
\label{thm2b}
f_{S}(x) = \frac{2a}{(1-a)x^3} \Big[ 2(x-a) \ln \frac{a-x}{a} + \frac{2a - ax - x}{a}\ln(1-x) \Big]
\end{equation}
using the method of partial fractions.
\end{Theorem}
Note that \eqref{thm2a} agrees with \eqref{density_a} if we let $k=2$ and make trivial substitutions. (It can be argued that sampling a particle alive at time $t$ 
and another particle alive at time $at$ independently in a tree conditioned to survive to time $t$ is equivalent to sampling two particles alive at time $at$ independently in a tree conditioned to survive to time $t$).

\begin{figure}[htbp]
\begin{center}
\includegraphics[scale=0.9]{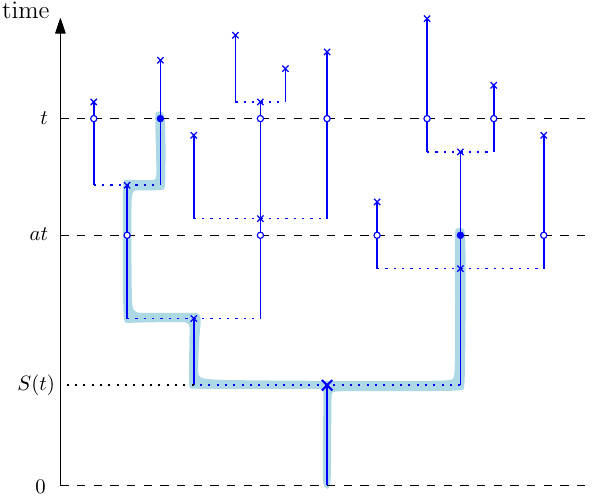}
\caption{Split times of two particles sampled at times $t$ and $at$.}
\end{center}
\end{figure}
Let us now recall a well-known one-to-one correspondence between a finite tree and its contour function (See, for example, \cite{L10}). The contour function drawn with a red line in Fig. \ref{contour} below is constructed in the obvious way by traversing the tree from left to right at linear speed $1$. The tree can be recovered from the contour function by identifying opposite points on the contour function along horizontal levels.
\begin{figure}[!htbp]
\begin{center}
\includegraphics[scale=0.75]{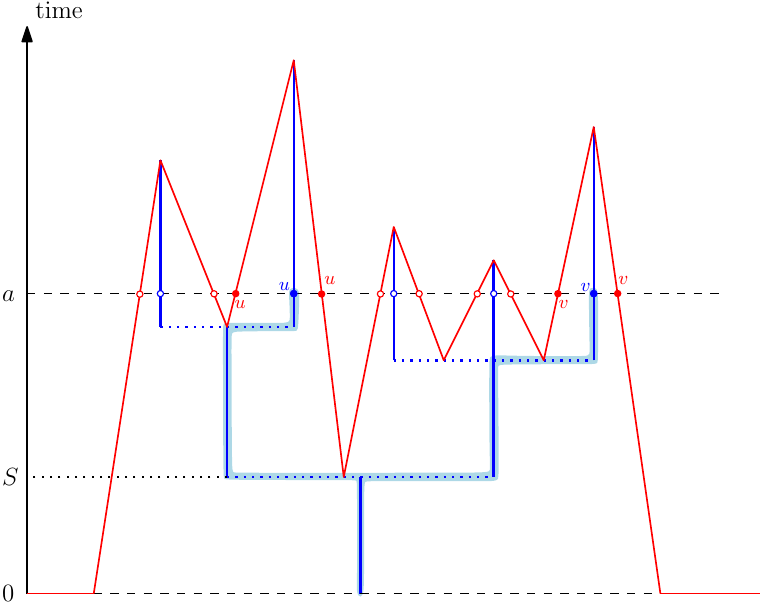}
\caption{A finite tree, its contour function and the split time of particles $u$ and $v$.}
\label{contour}
\end{center}
\end{figure}$ $

In particular, each particle alive at time $a$ in the tree corresponds to two intersections of level $a$ by the contour function and so 
\begin{equation}
\label{local_time}
N_a = \frac{1}{2}L^a,
\end{equation}
where $N_a$ is the number of particles in the tree alive at time $a$ and $L^a$ is the total local time at level $a$ accumulated by the contour function (that is, the number of visits of level $a$ by the contour function).

Moreover, the split time of any two particles in the tree would correspond to the minimum of the contour function between the corresponding intersections of the contour function and level $a$ (note that there are two points of intersection corresponding to each particle, but whichever one we choose will not affect the minimum of the excursion).

We now recall that the critical Galton-Watson tree conditioned to survive to time $t$ whose height is scaled by a factor of $t$ converges to a Continuum Random Tree (as established in the celebrated work of Aldous \cite{A91a}) while its contour function whose height and width are scaled by factors of $t$ and $\frac{v^2}{2}t$ respectively converges to a Brownian excursion conditioned to go above level $1$ (see e.g. \cite{L10}). In fact, recently Horton and Powell in \cite{HP26} have shown that this is  true in general for a very wide class of critical branching Markov processes.

Our main result therefore offers an efficient way of establishing the limiting joint law of split times of particles sampled from a critical branching process provided that the change of order in sampling and taking the limit is justified.
%====================================================================================================
\section{Proof of the Main Result}
In order to prove Theorem \ref{main} we are going to use the following two results about i.i.d. samples of uniform random variables on an interval of an exponentially-distributed length, which are of interest in their own right. 
\begin{Proposition}
\label{exp}
Suppose we are given a random variable $L \sim Exp (\lambda)$ for some $\lambda>0$ under some probability measure $\mathbb{P}$ and suppose that under $\mathbb{P}( \cdot | L)$ random variables  $U_i$, $i = 1, ..., k$ are independent and have $U[0,L]$ distribution. Let $L_i := U_{(i)}$, $i=1, ..., k$ be the order statistics of the sample $U_1$, ..., $U_k$ and let $L_0 := 0$, $L_{k+1} := L$.
\begin{figure}[htbp]
\begin{center}
\includegraphics[scale=1]{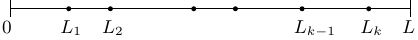}
\end{center}
\end{figure}$ $\newline
Define
\[
J_k(\theta_0, \theta_1, ..., \theta_k) := \mathbb{E} \Big( \prod_{i=0}^k \mathrm{e}^{-\theta_i (L_{i+1} - L_i)} \Big)
\]
to be the joint Laplace transform of the gaps $L_{i+1} - L_i$, $i =0$, ..., $k$. Then
\begin{equation}
\label{eq_exp}
J_k(\theta_0, \theta_1, ..., \theta_k) = k \lambda \int_0^\infty \frac{1}{\lambda + s + \theta_0} \cdot \Big[ \prod_{i=1}^{k-1} \frac{s}{\lambda + s + \theta_i} \Big] \cdot \frac{1}{\lambda + s + \theta_k} \mathrm{d} s.
\end{equation}
\end{Proposition}

\begin{Proposition}
\label{two_exp}
Suppose we are given independent random variables $L \sim Exp (\lambda)$ and $\tilde{L} \sim Exp (\mu)$ for some $\lambda>0$ and $\mu>0$ and let $L_0$, ..., $L_{k+1}$ be as in Proposition \ref{exp} above. Define 
\[
\tilde{J}_k(\theta_0, \theta_1, ..., \theta_k) := \mathbb{E} \Big( \prod_{i=0}^k \mathrm{e}^{-\theta_i (L_{i+1} - L_i)} \Big\vert L > \tilde{L} \Big)
\]
to be the joint Laplace transform of the gaps $L_{i+1} - L_i$, $i =0$, ..., $k$ conditional on $L > \tilde{L} $. Then
\begin{equation}
\label{eq_two_exp}
\tilde{J}_k(\theta_0, \theta_1, ..., \theta_k) = \frac{\lambda + \mu}{\mu} \Big( J_k(\theta_0, \theta_1, ..., \theta_k) \ - \ 
J_k(\theta_0 \!+\! \mu, \ \theta_1 \!+\! \mu, \ ..., \ \theta_k \!+\! \mu) \Big),
\end{equation}
where $J_k$ is the unconditional Laplace transform defined in Proposition \ref{exp} above.
\end{Proposition}
The proofs of Proposition \ref{exp} and \ref{two_exp} will be given in the last section. 
Let us now present the proof of the main result. Namely that if we sample $k$ points of intersection 
of level $a \in (0,1]$ by a Brownian excursion from level $0$ conditioned to have height at least $1$ then the joint distribution of the minima of the excursion between neighbouring sampled points is given by \eqref{tail_a} and \eqref{density_a}. 

Note that \eqref{density_a} follows from \eqref{tail_a} directly since 
\[
f_{S_1, ..., S_{k-1}}(x_1, ..., x_{k-1}) = (-1)^{k-1} \frac{\partial^k}{\partial x_1 \cdots \partial x_{k-1}} \mathbb{P} \big(S_1 > x_1, \ ..., \ S_{k-1} > x_{k-1} \big).
\]
Thus we only need to prove \eqref{tail_a}. We are going to do this for the case $a=1$ first and then for the case $a \in (0,1)$.

\begin{proof}[Proof of Theorem \ref{main} ($a=1$ case)]
We consider a Brownian motion $W_t$, $t \geq 0$ with $W_0 = 1$ and we let $T_0$ be the time when it hits $0$ for the first time. In this way, $W_t$, $t \in [0, T_0]$ would correspond to the part of a Brownian excursion $X$ from \eqref{x} that makes contribution to the local time at level $1$.

We let $L^1_t$, $t \geq 0$ be the local time at level $1$ of the process $W$ and we let
\[
L := L^1_{T_0}
\]
be the amount of local time accumulated by $W$ over the time interval $[0, T_0]$. It is well-known (from either It\^o's excursion theory or the Ray-Knight theorem) that 
\[
L \sim Exp \Big(\frac{1}{2} \Big).
\]
We let $U_1$, ..., $U_k$ be random variables, which conditional on $(L^1_t)_{t \geq 0}$ are independent and have distribution 
\[
\mathbb{P} \Big( U_i \leq t \big\vert (L^1_t)_{t \geq 0} \Big) = \frac{L^1_t}{L^1_{T_0}} = \frac{L^1_t}{L},
\]
$t \in [0, T_0]$. Thus $L^1_{U_1}$, ..., $L^1_{U_k}$ are independent $U[0,L]$ random variables under $\mathbb{P}(\cdot | L)$.

We let
\[
L_i := L^1_{U_{(i)}}
\] 
where $U_{(1)}$, ..., $U_{(k)}$ are the order statistics of $U_1$, ..., $U_k$.

Finally, we let
\[
S_i := \min_{t \in [U_{(i)}, U_{(i+1)}]} W_t ,
\]
$i = 1, 2, \cdots, k-1$.

\begin{figure}[htbp]
\begin{center}
\includegraphics[scale=1.1]{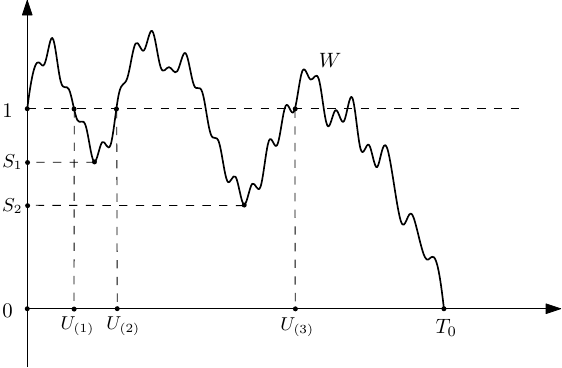}
\caption{Minima between points of intersection of an excursion of level $1$.}
\end{center}
\end{figure}

We need to show that $S_1$, ..., $S_{k-1}$ do indeed satisfy \eqref{tail_a}. Let us take any $x_1$, ..., $x_{k-1} \in (0,1]$. Then for $i \in \{1, ..., k-1\}$ the event $\{S_i \leq x_i\}$ is equivalent to the event of $W$ having an excursion from level $1$ below level $x_i$ but not below level $0$ over the interval $[L_i, L_{i+1}]$ on the local time scale. According to It\^o's excursion theory such excursions arrive according to a Poisson process with rate 
\[
n_i = \frac{1}{2(1-x_i)} - \frac{1}{2}
\] 
on the local time scale at level $1$ (see, for example, \cite{R89}). Hence
\begin{equation}
\label{laplace}
\mathbb{P} \big(S_1 > x_1, \ ..., \ S_{k-1} > x_{k-1} \big) = \mathbb{E} \Big[ \mathrm{e}^{-(L_2 - L_1) n_1} \cdots 
\mathrm{e}^{-(L_k - L_{k-1})n_{k-1}} \Big].
\end{equation}
Since $L_1$, ..., $L_{k-1}$ are the order statistics of $k$ independent $U[0,L]$ random variables where $L \sim Exp(\frac{1}{2})$ we can apply identity \eqref{eq_exp} with $\lambda = \frac{1}{2}$, $\theta_0 = \theta_k = 0$ and $\theta_i = n_i$, $i=1$, ..., $k-1$ to get
\begin{align*}
\mathbb{P} \big(S_1 > x_1, \ ..., \ S_{k-1} > x_{k-1} \big) = &J_k (0, n_1, ..., n_{k-1}, 0)\\
= &\frac{k}{2} \int_0^\infty \Big( \frac{1}{\frac{1}{2} +s} \Big)^2 \prod_{i=1}^{k-1} \frac{s}{s + \frac{1}{2(1-x_i)}} \mathrm{d}s\\
= &2k \int_0^\infty \Big( \frac{1}{1 +2s} \Big)^2 \prod_{i=1}^{k-1} \frac{2s(1-x_i)}{1 + 2s(1-x_i)} \mathrm{d}s\\
= &k \int_0^\infty \Big( \frac{1}{1 +s} \Big)^2 \prod_{i=1}^{k-1} \frac{s(1-x_i)}{1 + s(1-x_i)} \mathrm{d}s
\end{align*}
making the substitution $2s=s$ in the last line. This proves \eqref{tail_a} in the case $a=1$.
\end{proof}

\begin{proof}[Proof of Theorem \ref{main} ($a \in (0,1)$ case)]
We now consider a Brownian motion $W_t$, $t \geq 0$ with $W_0 = a \in (0,1)$ and we let $T_0$ and $T_1$ be the first hitting times of levels $0$ and $1$ respectively. In this way, conditional on $\{T_1 < T_0\}$ , $W_t$, $t \in [0, T_0]$ would correspond to the part of the Brownian excursion $X$ from \eqref{x} that makes contribution to the local time at level $a$. 

We let $L^a_t$, $t \geq 0$ be the local time at level $a$ of the process $W$ and we let 
\[
L :=L^a_{T_0}
\]
be the amount of local time at level $a$ accumulated by $W$ over the time interval $[0, T_0]$. We also let 
\[
\tilde{L} := L^a_{T_1}
\]
be the amount of local time at level $a$ accumulated by $W$ over the time interval $[0, T_1]$.

Since excursions from level $a$ above level $1$ and below level $0$ arrive independently at rates $\frac{1}{2(1-a)}$ and $\frac{1}{2a}$ respectively we have that
\[
L \sim Exp \Big(\frac{1}{2a} \Big) \ \text{ and } \ \tilde{L} \sim Exp \Big(\frac{1}{2(1-a)} \Big)
\]
independently. Moreover, 
\[
\{ T_1 < T_0 \} \equiv \{ \tilde{L} < L \}.
\]
We let $U_1$, ..., $U_k$ be random variables, which conditional on $(L^a_t)_{t \geq 0}$ are independent and have distribution 
\[
\mathbb{P} \Big( U_i \leq t \big\vert (L^a_t)_{t \geq 0} \Big) = \frac{L^a_t}{L^a_{T_0}} = \frac{L^a_t}{L},
\]
$t \in [0, T_0]$. Thus $L^a_{U_1}$, ..., $L^a_{U_k}$ are independent $U[0,L]$ random variables under $\mathbb{P}(\cdot | L)$.

We let
\[
L_i := L^a_{U_{(i)}}
\] 
where $U_{(1)}$, ..., $U_{(k)}$ are the order statistics of $U_1$, ..., $U_k$.

Finally, we let
\[
S_i := \min_{t \in [U_{(i)}, U_{(i+1)}]} W_t ,
\]
$i = 1, 2, \cdots, k-1$.

\begin{figure}[htbp]
\begin{center}
\includegraphics[scale=1]{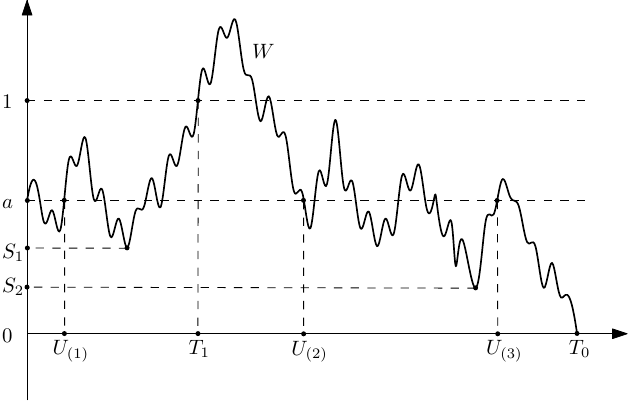}
\caption{Minima between points of intersection of an excursion of level $a$.}
\end{center}
\end{figure}
We need to show that $S_1$, ..., $S_{k-1}$ do indeed satisfy \eqref{tail_a}. Let us take any $x_1$, ..., $x_{k-1} \in (0,a]$. Then for $i \in \{1, ..., k-1\}$ the event $\{S_i \leq x_i\}$ is equivalent to the event of $W$ having an excursion from level $a$ below level $x_i$ but not below level $0$ over the interval $[L_i, L_{i+1}]$ on the local time scale. Such excursions arrive according to a Poisson process of rate 
\[
n_i = \frac{1}{2(a-x_i)} - \frac{1}{2a}
\] 
on the local time scale at level $a$. We also need to condition on the event $\{\tilde{L} < L\}$. We therefore need to calculate
\begin{equation}
\label{laplace_a}
\mathbb{P} \big(S_1 > x_1, \ ..., \ S_{k-1} > x_{k-1} \big\vert \tilde{L} < L \big) = \mathbb{E} \Big[ \mathrm{e}^{-(L_2 - L_1) n_1} \cdots 
\mathrm{e}^{-(L_k - L_{k-1})n_{k-1}} \Big\vert \tilde{L} < L\Big].
\end{equation}
Since $L_1$, ..., $L_{k-1}$ are the order statistics of $k$ independent $U[0,L]$ random variables where $L \sim Exp(\frac{1}{2a})$ we can apply identity \eqref{eq_two_exp} with $\lambda = \frac{1}{2a}$, $\mu = \frac{1}{2(1-a)}$, $\theta_0 = \theta_k = 0$ and $\theta_i = n_i$, $i=1$, ..., $k-1$ to get

\begin{align*}
\mathbb{P} \big(S_1 \!>\! x_1, \ ..., \ S_{k-1} \!>\! x_{k-1} \big\vert \tilde{L} < L \big) = &\tilde{J}_k(0, n_1, ..., n_{k-1}, 0)\\
= &\frac{\lambda\!+\!\mu}{\mu} \Big( J_k(0 , n_1, ..., n_{k-1}, 0 ) - J_k(\mu , n_1 \!+\! \mu, ..., n_{k-1} \!+\! \mu, \mu ) \Big)\\
= &\frac{k \lambda (\lambda\!+\!\mu)}{\mu} \int_0^\infty \frac{1}{(\lambda\!+\!s)^2} \prod_{i=1}^{k-1} \frac{s}{\lambda \!+\! s \!+\! n_i}\\ 
&\qquad\qquad\qquad\qquad - \frac{1}{(\lambda\!+\!s\!+\!\mu)^2} \prod_{i=1}^{k-1} \frac{s}{\lambda \!+\! s \!+\! \mu \!+\! n_i} \mathrm{d}s\\
= &\frac{k}{2a^2} \int_0^\infty \frac{1}{(\frac{1}{2a}+\!s)^2} \prod_{i=1}^{k-1} \frac{s}{s \!+\! \frac{1}{2(a-x_i)}}\\ 
&\qquad\qquad\qquad\qquad - \frac{1}{(\frac{1}{2a(1-a)} + s)^2} \prod_{i=1}^{k-1} \frac{s}{s \!+\! \frac{1}{2(1-a)} \!+\! \frac{1}{2(a-x_i)}} \mathrm{d}s\\
= &2k \int_0^\infty \frac{1}{(1+2as)^2} \prod_{i=1}^{k-1} \frac{2(a-x_i)s}{1 + 2(a-x_i)s} \mathrm{d}s\\ 
&\qquad - 2k \int_0^\infty \frac{(1-a)^2}{(1 + 2a(1-a)s)^2} \prod_{i=1}^{k-1} \frac{2(1-a)(a-x_i)s}{1 - x_i + 2(1-a)(a-x_i)s} \mathrm{d}s
\end{align*}
Making the substitution $2s = s$ in the first integral and $2(1-a)s = s$ in the second integral yields the result.
\end{proof}

%====================================================================================================
\section{Proof of Propositions \ref{exp} and \ref{two_exp}}

\begin{proof}[Proof of Proposition \ref{exp}]
Suppose $L \sim Exp (\lambda)$ under a probability measure $\mathbb{P}$. Let us define the new probability measure $\mathbb{Q}_k$ via the Radon-Nikodym derivative:
\begin{equation}
\label{radon_nikodym}
\frac{\mathrm{d} \mathbb{Q}_k}{\mathrm{d} \mathbb{P}} := \frac{\lambda^k}{k!} L^k.
\end{equation}
Let $L_0$, $L_1$, ..., $L_k$, $L_{k+1}$ be as above and define 
\[
I_k(\theta_0, \theta_1, ..., \theta_k) := \mathbb{Q}_k \Big( \prod_{i=0}^k \mathrm{e}^{-\theta_i (L_{i+1} - L_i)} \Big)
\]
to be the joint Laplace transform of the gaps $L_{i+1} - L_i$, $i =0$, ..., $k$ under $\mathbb{Q}_k$. Then 
\begin{equation}
\label{I}
I_k(\theta_0, \theta_1, ..., \theta_k) = \prod_{i=0}^k \frac{\lambda}{\lambda + \theta_i} 
\end{equation}
In other words, under $\mathbb{Q}_k$, the gaps $L_{i+1} - L_i$, $i =0$, ..., $k$ are independent $Exp(\lambda)$ random variables.

The easiest way to verify identity \eqref{I} is by induction. Indeed, if $k=0$ then $\mathbb{Q}_0 \equiv \mathbb{P}$, $L_1 = L$ and so we have
\[
I_0(\theta_0) = \mathbb{Q}_0 \big( \mathrm{e}^{-\theta_0(L_1 - L_0)} \big) = \mathbb{E} \big( \mathrm{e}^{-\theta_0 L} \big) = \frac{\lambda}{\lambda + \theta_0},
\] 
where the last equality follows from the Laplace transform formula for $L \sim Exp (\lambda)$.

If we assume that \eqref{I} holds for some $k \geq 0$ then applying the change of measure, conditioning on $L$ and substituting the joint probability density function of $L_1$, ..., $L_{k+1}$ under $\mathbb{P}(\cdot | L)$ gives the following expression for $I_{k+1}$:
\begin{align*}
I_{k+1}(\theta_0, \theta_1, ..., \theta_{k+1})  =&\mathbb{Q}_{k+1} \Big[ \prod_{i=0}^{k+1} \mathrm{e}^{-\theta_i (L_{i+1} - L_i)} \Big]\\
= &\mathbb{E} \Big[ \frac{\lambda^{k+1}}{(k+1)!} L^{k+1} \prod_{i=0}^{k+1} \mathrm{e}^{-\theta_i (L_{i+1} - L_i)} \Big]\\
=& \mathbb{E} \Big[ \mathbb{E} \Big( \frac{\lambda^{k+1}}{(k+1)!} L^{k+1} \prod_{i=0}^{k+1} \mathrm{e}^{-\theta_i (L_{i+1} - L_i)} \Big\vert L \Big) \Big]\\
= &\mathbb{E} \Big[ \frac{\lambda^{k+1}}{(k+1)!} L^{k+1} \int_{0 \leq x_1 \leq \cdots \leq x_{k+1} \leq L} \Big(
\prod_{i=0}^{k+1} \mathrm{e}^{-\theta_i (x_{i+1} - x_i)} \Big) \ \frac{(k+1)!}{L^{k+1}} \ \mathrm{d}(x_1, x_2, \cdots, x_{k+1}) \Big]\\
= &\lambda^{k+1} \mathbb{E} \Big[ \int_{0 \leq x_1 \leq \cdots \leq x_{k+1} \leq L} 
\prod_{i=0}^{k+1} \mathrm{e}^{-\theta_i (x_{i+1} - x_i)} \ \mathrm{d}(x_1, x_2, \cdots, x_{k+1}) \Big]
\end{align*}
(where $L_{k+2} = x_{k+2} = L$).

Then integrating out $x_1$ yields
\begin{align*}
I_{k+1}(\theta_0, \theta_1, \cdots, \theta_{k+1}) = &\lambda^{k+1} \mathbb{E} \Big[ \int_{0 \leq x_2 \leq \cdots \leq x_{k+1} \leq L} 
\Big( \int_0^{x_2} \mathrm{e}^{- (\theta_0 - \theta_1)x_1} \mathrm{d}x_1 \Big) \mathrm{e}^{-\theta_1 x_2} \\
& \qquad\qquad\qquad\qquad\qquad\qquad\quad \prod_{i=2}^{k+1} \mathrm{e}^{-\theta_i (x_{i+1} - x_i)} \ 
\mathrm{d}(x_2, \cdots, x_{k+1}) \Big]\\
= &\lambda^{k+1} \mathbb{P} \Big[ \int_{0 \leq x_2 \leq \cdots \leq x_{k+1} \leq L} 
\Big( \frac{1 - \mathrm{e}^{-(\theta_0 - \theta_1)x_2}}{\theta_0 - \theta_1} \Big) \mathrm{e}^{-\theta_1 x_2}\\
& \qquad\qquad\qquad\qquad\qquad\qquad\quad \prod_{i=2}^{k+1} \mathrm{e}^{-\theta_i (x_{i+1} - x_i)} \ 
\mathrm{d}(x_2, \cdots, x_{k+1}) \Big]\\
= &\frac{\lambda^{k+1}}{\theta_0 - \theta_1}  \mathbb{P} \Big[ \int_{0 \leq x_2 \leq \cdots \leq x_{k+1} \leq L} 
\Big( \mathrm{e}^{- \theta_1 (x_2 - 0)} - \mathrm{e}^{- \theta_0 (x_2 - 0)} \Big)\\
& \qquad\qquad\qquad\qquad\qquad\qquad\quad \prod_{i=2}^{k+1} \mathrm{e}^{-\theta_i (x_{i+1} - x_i)} \ 
\mathrm{d}(x_2, \cdots, x_k) \Big]\\
= & \frac{\lambda}{\theta_0 - \theta_1} \Big( I_k(\theta_1, \theta_2, \cdots, \theta_{k+1}) - I_k(\theta_0, \theta_2, \cdots, \theta_{k+1}) \Big).
\end{align*}
Then from the induction hypothesis it follows that 
\begin{align*}
I_{k+1}(\theta_0, \theta_1, \cdots, \theta_{k+1}) = &\frac{\lambda}{\theta_0 - \theta_1} \Big( \frac{\lambda}{\lambda + \theta_1} \cdot \frac{\lambda}{\lambda + \theta_2} \cdots \frac{\lambda}{\lambda + \theta_{k+1}} \ - \ \frac{\lambda}{\lambda + \theta_0} \cdot \frac{\lambda}{\lambda + \theta_2} \cdots \frac{\lambda}{\lambda + \theta_{k+1}} \Big)\\
=&\frac{\lambda}{\theta_0 - \theta_1} \Big( \frac{\lambda}{\lambda + \theta_1}  - \frac{\lambda}{\lambda + \theta_0} \Big) \prod_{i=2}^{k+1} \frac{\lambda}{\lambda + \theta_i}\\
=&\prod_{i=0}^{k+1} \frac{\lambda}{\lambda + \theta_i}. 
\end{align*}
Lets us now establish identity \eqref{eq_exp}. From basic calculus we have that
\begin{equation}
\label{reciprocal}
\frac{1}{L^k} = \frac{1}{(k-1)!} \int_0^\infty s^{k-1} \mathrm{e}^{- s L} \mathrm{d}s = \frac{1}{(k-1)!}
\int_0^\infty s^{k-1} \prod_{i=0}^k \mathrm{e}^{- s (L_{i+1} - L_i)} \mathrm{d}s.
\end{equation}
Then applying the change of measure \eqref{radon_nikodym}, identity \eqref{reciprocal}, Fubini's theorem and \eqref{I} we get 
\begin{align*}
J_k(\theta_0, \theta_1, \cdots , \theta_k) =&\mathbb{E} \Big[ \prod_{i=0}^k \mathrm{e}^{-\theta_i(L_{i+1} - L_i) }   \Big]\\
= &\mathbb{Q}^k \Big[ \frac{k!}{\lambda^k} \frac{1}{L^k} \prod_{i=0}^k \mathrm{e}^{-\theta_i(L_{i+1} - L_i) } \Big]\\
= &\frac{k!}{\lambda^k} \mathbb{Q}^k \Big[ \frac{1}{(k-1)!} \int_0^\infty s^{k-1} \prod_{i=0}^k \mathrm{e}^{-(s + \theta_i) (L_{i+1} - L_i) }   \mathrm{d}s \Big]\\
= &\frac{k}{\lambda^k} \int_0^\infty s^{k-1} I_k(s+\theta_0, \ s+\theta_1, \ \cdots , \ s+\theta_k) \mathrm{d} s\\
= & k \lambda \int_0^\infty \frac{1}{\lambda + s + \theta_0} \cdot \Big[ \prod_{i=1}^{k-1} \frac{s}{\lambda + s + \theta_i} \Big] \cdot \frac{1}{\lambda + s + \theta_k} \mathrm{d} s.
\end{align*}

\end{proof}

\begin{proof}[Proof of Proposition \ref{two_exp}]
Suppose $L \sim Exp(\lambda)$ and $\tilde{L} \sim Exp(\mu)$ are independent under a probability measure $\mathbb{P}$. Let $L_1$, ..., $L_k$, as before, be the order statistics of $k$ independent $U[0, L]$ random variables under $\mathbb{P}(\cdot | L)$. Then
\begin{align}
\label{J}
\tilde{J}_k(\theta_0, \theta_1, \cdots , \theta_k) =&\mathbb{E} \Big[ \prod_{i=0}^k \mathrm{e}^{-\theta_i(L_{i+1} - L_i)} \Big\vert L > \tilde{L}  \Big] \nonumber\\
= &\frac{1}{\mathbb{P} (L > \tilde{L})} \mathbb{E} \Big[ \prod_{i=0}^k \mathrm{e}^{-\theta_i(L_{i+1} - L_i)} \ \mathbf{1}_{ \{ L > \tilde{L}\}} \Big] \nonumber\\
= &\frac{\lambda + \mu}{\mu} \mathbb{E} \Big[ \mathbb{E} \Big( \prod_{i=0}^k \mathrm{e}^{-\theta_i(L_{i+1} - L_i)} \ \mathbf{1}_{ \{ L > \tilde{L}\}} \Big\vert L, L_1, ..., L_k\Big) \Big] \nonumber\\
= &\frac{\lambda + \mu}{\mu} \mathbb{E} \Big[ \prod_{i=0}^k \mathrm{e}^{-\theta_i(L_{i+1} - L_i)} \ \mathbb{P} \big( \tilde{L} < L \big\vert L, L_1, ..., L_k\big) \Big].
\end{align}
We note that
\begin{equation}
\label{J2}
\mathbb{P} \big( \tilde{L} < L \big\vert L, L_1, ..., L_k\big) =1 - \mathrm{e}^{-\mu L} = 1 - \prod_{i=0}^k \mathrm{e}^{-\mu (L_{i+1} - L_i)}
\end{equation}
Substituting \eqref{J2} into \eqref{J} gives 
\begin{align*}
\tilde{J}_k(\theta_0, \theta_1, \cdots , \theta_k)
= &\frac{\lambda + \mu}{\mu} \mathbb{E} \Big[ \prod_{i=0}^k \mathrm{e}^{-\theta_i(L_{i+1} - L_i)} \ \Big( 1 - \prod_{i=0}^k \mathrm{e}^{-\mu (L_{i+1} - L_i)} \Big) \Big]\\
= &\frac{\lambda + \mu}{\mu} \mathbb{E} \Big[ \prod_{i=0}^k \mathrm{e}^{-\theta_i(L_{i+1} - L_i)} -  \prod_{i=0}^k \mathrm{e}^{-(\theta_i + \mu)(L_{i+1} - L_i)} \Big]\\
= &\frac{\lambda + \mu}{\mu} \Big( J_{k}(\theta_0, \theta_1, \cdots, \theta_k) - J_{k}(\theta_0 \!+\! \mu, \ \theta_1 \!+\! \mu, \ \cdots, \ \theta_k \!+\! \mu) \Big).
\end{align*}
\end{proof}

%====================================================================================================


\begin{thebibliography}{99}
\bibitem{A91a} D. Aldous``The Continuum Random Tree I", Annals of Probability. \textbf{19(1)}, 1-28 (1991)
\bibitem{A91b} D. Aldous``The Continuum Random Tree II: an overview", in Stochastic Analysis: Proceedings of the Durham Symposium on Stochastic Analysis, 23-70 Cambridge University Press (1991)
\bibitem{A93} D. Aldous``The Continuum Random Tree I", Annals of Probability. \textbf{21(1)}, 248-289 (1993)
\bibitem{HJR20} S.C. Harris, S.G.G. Johnston, M.I. Roberts ``The coalescent structure of continuous-time Galton-Watson trees", Annals of Applied Probability. \textbf{30(3)}, 1368–1414 (2020)
\bibitem{HHKP24} S.C. Harris, E. Horton, A.E. Kyprianou, E. Powell ``Many-to-few for non-local branching Markov process", Electronic Journal of Probability. \textbf{29}, article no.41, 1-26 (2024)
\bibitem{HP26} E. Horton, E. Powell ``Convergence to the Brownian CRT for critical branching Markov processes", arXiv preprint  arXiv:2601.05906v2 (2026)
\bibitem{L10} J.-F. Le Gall ``It\^o's excursion theory and random trees", Stochastic Processes and their Applications. \textbf{120(5)}, 721-749 (2010)
\bibitem{NP89} J.Neveu, J.W. Pitman ``The branching process in a brownian excursion", S\'eminaire de probabilit\'es (Strasbourg). \textbf{23}, 248-257 (1989)
\bibitem{P04} L. Popovic ``Asymptotic genealogy of a critical branching process", Annals of Applied Probability. \textbf{14(4)}, 2120-2148 (2004)
\bibitem{R89} L.C.G. Rogers ``A Guided Tour through Excursions", Bulletin of the London Mathematical Society. \textbf{21(4)}, 305-341 (1989)
\bibitem{Z75} A.M. Zubkov ``Limiting Distributions of the Distance to the Closest Common Ancestor", Theory of Probability and its Applications. \textbf{20}, 602-612 (1975) 

\end{thebibliography}
\end{document}